\documentclass[a4paper,12pt]{amsart}
\usepackage{massive-preambule}
\usepackage{amssymb,amscd}

\begin{document}

\title{Persistent massive attractors of smooth maps}

\author{{D.~Volk}}
%    Address of record for the research reported here
\address{Denis~Volk
\newline\hphantom{iii} Scuola Internazionale Superiore di Studi Avanzati
\newline\hphantom{iii} Institute for Information Transmission Problems, Russian Academy of Sciences}
%\address{}
%    Current address
%\curraddr{SISSA, via Bonomea n.265, 34136 Trieste Italy}
\email{denis.volk@sissa.it}
%    \thanks will become a 1st page footnote.
\thanks{The author was supported in part by grants RFBR 10-01-00739, RFBR/CNRS 10-01-93115, President's of Russia MK-2790.2011.1, PRIN, ``Young SISSA Scientists''}

%    Information for second author
%\author{Author Two}
%\address{Mathematical Research Section, School of Mathematical Sciences,
%Australian National University, Canberra ACT 2601, Australia}
%\email{two@maths.univ.edu.au}
%\thanks{Support information for the second author.}

%    General info

%\date{\timestamp}

%\dedicatory{This paper is dedicated to our advisors.}


\begin{abstract}
For a smooth manifold of any dimension greater than one, we present an open set of smooth endomorphisms such that any of them has a transitive attractor with a non-empty interior. These maps are $m$-fold non-branched coverings, $m \ge 3$. The construction applies to any manifold of the form $S^1 \times M$, where $S^1$ is the standard circle and $M$ is an arbitrary manifold.
\end{abstract}

\maketitle

\section{Introduction}  \label{s:intro}

Consider a smooth map $F$ of a compact manifold $X$ (possibly with boundary) into itself. The semigroup generated by $F$ is a discrete-time dynamical system. In the theory of dynamical systems, we are interested in the limit behavior of the orbits. In a wide class of dynamical systems, one can find a proper subset $\Lambda \subset X$ and an open set $B \supset \Lambda$ such that the trajectories of almost every point in $B$ accumulate (in some sense) onto $\Lambda$. Then $\Lambda$ is said to be an \emph{attractor} and $B$ its \emph{basin}. Various formalizations of the word ``accumulate'' give rise to different (usually non-equivalent) definitions of attractors. The proper identification of an attractor is very important in applications because it is often possible to reduce the dimensions or size of the system by restricting it to its attractor.

We do not consider the trivial case, $\Lambda = X$. In this situation, the notion of an attractor does not provide any additional information on the system. Now we address the following natural question suggested by Yu.~Ilyashenko: if $\Lambda \neq X$, how \emph{big} can $\Lambda$ be?

There are many ways to tell what is ``big''. An incomplete list of approaches includes (listed from strongest to weakest)
\begin{itemize}
\item having non-empty interior (\emph{T-big}, for topology),
\item having a positive Lebesgue measure (\emph{M-big}, for measure)
\item having full Hausdorff dimension (\emph{D-big}, for dimension).
\end{itemize}

Given full freedom, one can construct dynamical systems with very weird behavior. Thus we are interested only in (counter)examples which are smooth enough and admit some sort of non-degeneracy. In this paper, we consider only the systems which are at least $C^{1+\alpha}$. When we speak of robustness, we assume the space of all systems is equipped with the $C^1$ topology.

This topic has been very hot in recent years. Abdenur, Bonatti, D\'{\i}az~\cite{Abdenur2004} conjectured that $C^1$-generic transitive diffeomorphisms whose non-wandering set has a non-empty interior are (globally) transitive. If this is true, no T-big attractors are possible for $C^1$-generic diffeomorphisms. They gave the proofs for three cases: hyperbolic diffeomorphisms, partially hyperbolic diffeomorphisms with two hyperbolic bundles, and tame diffeomorphisms. They mentioned that in the first case, the proof is folklore; in the second one, they adapted the proof of Brin~\cite{Brin1975}. On the other hand, Fisher~\cite{Fisher2006a} gave an interesting example of a hyperbolic set with robustly non-empty interior. Unfortunately, this set is not transitive. Finally, a result by Gorodetski~\cite{Gorodetski2012} shows that D-big transitive invariant sets robustly appear in one-parameter unfoldings of homoclinic tangencies.

However, in these papers the authors never assumed the non-wandering set to be an \emph{attractor}. A classical result by Bowen~\cite{Bowen1975a} states that every \emph{hyperbolic} attractor of a diffeomorphism has a zero Lebesgue measure, thus forbidding M-big hyperbolic attractors. But in the non-hyperbolic setting the question is still open. In~\cite{McCluskey1983}, McCluskey and Manning showed that when a $C^2$-diffeomorphism of a 2-surface bifurcates from an Anosov to a DA, the Hausdorff dimension of the newborn attractor is 2 (moreover, it is a continuous function of the parameter). There is also a positive result for the space of all boundary preserving diffeomorphisms of a compact manifold \emph{with boundary}. Ilyashenko~\cite{Ilyashenko2011} showed that there exists a quasiopen set of such maps such that any map in this set has a M-big attractor.

In the realm of smooth endomorphisms, there are more positive answers. In $\dim = 1$ in the quadratic family, the famous Feigenbaum attractor with an absolutely continuous invariant measure provides an example of a M-big attractor. Also, Dobbs~\cite{Dobbs2006} showed there exists a $C^2$-smooth map of the interval with a finite critical set whose Julia set (see~\cite{Melo1993}) has Hausdorff dimension 1. This is an example of a D-big repeller.

The goal of this paper is to provide a $C^1$-robust example of a T-big attractor of an endomorphism. Our construction is based on skew products over expanding circle maps.
%There are two noteworthy results closely related to this.
%The skew products over expanding circle maps is another interesting class of smooth endomorphisms.
Tsujii, Avila, Gou{\"e}zel~\cite{Tsujii2001}, \cite{Avila2006} and Rams~\cite{Rams2003} already showed that in certain concrete family of such skew products the attractor may stably carry an invariant SRB measure which is absolutely continuous w.r.t. the Lebesgue measure. Thus their attractors are M-big.

In the Tsujii's example, the fiber maps are uniformly contracting, which ensures the existence of an attractor. Viana~\cite{Viana1997} used another mechanism to create an attractor, namely, the unimodal maps in the fibers. The result is an M-big attractor exhibiting all positive Lyapunov exponents almost everywhere. Due to this expansion property, the Viana's attractor is also T-big as shown in~\cite{Alves2002}. Moreover, it is $C^3$-robust.

%In the next Section, we present two versions of our main theorem, each one deals with its specific space of dynamical systems.
In our main Theorem~\ref{t:endos}, we also construct an example of a T-big attractor. However, the argument is more in the Tsujii's spirit thus showing that a single expanding direction is already enough to create robust T-big attractors. Moreover, our example is robust in the space of $C^1$-smooth endomorphisms. The skew products from~\cite{Tsujii2001}, \cite{Avila2006} and~\cite{Rams2003} can be included in this space as subsets of infinite codimension.

Our proof is based on the theory of partially hyperbolic perturbations started by Hirsch, Pugh, Shub~\cite{HPS1977}, and continued by Gorodetski~\cite{Gorodetskiui2006} and Ilyashenko, Negut~\cite{Ilyashenko2012}. First we construct a certain skew product over an expanding circle map. Then we prove that every skew product which
%smooth map which
is close enough to it, has a massive attractor (Theorem~\ref{t:skpr-circle}). Finally, we show the same is true for endomorphisms which are not necessary skew products (Theorem~\ref{t:endos}).

\section{Main theorems} \label{s:main-theorems}

Let $X$ be some smooth manifold and let a smooth map~$F\colon X\to X$ define a discrete-time dynamical system. In this paper, we use the term ``region'' to describe the closure of a non-empty connected open set in $X$. First we recall some classical definitions related to attractors. Denote by $\inter A$ the interior of any set $A$.

\begin{definition}  \label{d:amax}
A compact set $D \subset X$ is called \emph{trapping} if ${F(D)} \subset \inter D$. Then the closed $F$-invariant set
$$
\Amax := \bigcap_{n=0}^\infty F^n(D)
$$
is said to be a \emph{maximal attractor} for $F$.
\end{definition}
This definition gives the description of an attractor from a geometrical point of view. For the statistical description, we recall the following
\begin{definition}  \label{d:srb}
An $F$-invariant measure $\mu$ is called \emph{Sinai-Ruelle-Bowen (SRB)} if there exists a measurable set~$E \subset X$, $\Leb(E) > 0$, such that for any test function $\psi \in C(X)$ and any $x \in E$ we have
$$
\lim\limits_{n\to\infty} \frac{1}{n}\sum_{i=0}^{n-1} \psi (F^i(x)) = \int\limits_{X} \psi \, d\mu.
$$
The set $E = E(\mu)$ is called the \emph{basin} of $\mu$.
\end{definition}

\begin{definition}  \label{d:ma}
Let $\Amax \ne X$ be a maximal attractor for a certain trapping region~$D \supset \Amax$. We say that $\Amax$ is \emph{massive} if
\begin{enumerate}
\item $\Amax$ is a region;
\item $\Amax$ is the support of an invariant ergodic SRB measure $\mu$ such that $E(\mu) \supset D$.
\end{enumerate}
\end{definition}
In particular, any massive attractor is T-big and transitive.

Now let $M$ be any compact Riemannian manifold,
%\footnote{Compact or not, with boundary or not, oriented or not.}
and fix any $m \in \bbN$, $m \ge 3$. The existence theorems of this paper will be proven in a constructive way. The dimension of $M$ and the number~$m$ are the main parameters of the construction.

Let $S^1 = \bbR / \bbZ$ be the standard circle and let $X := S^1 \times M$. Denote by $C^1(X)$ the space of all $C^1$-smooth $m$-to-$1$ coverings of $X$ by itself, with $C^1$-topology.

\begin{theorem} \label{t:endos}
There exists a nonempty open set~$\mathcal O \subset C^1(X)$ such that any $\FF \in \mathcal O$ has a massive attractor $\Amax(\FF)$.
%Moreover, the set $U$ is uniform for all these dynamical systems and has the form $U = S^1 \times B$ where $B \subset M$ is diffeomorphic to an open ball.
\end{theorem}

\begin{remark}
In our construction, any covering~$\FF$ is a factor of an invertible dynamical system $\hat \FF$. Any such $\hat \FF$ has an invariant hyperbolic set $\Lambda (\hat \FF)$ and $\Amax(\FF)$ is the projection of this set. In particular, this means $\mu$ itself is hyperbolic: its Lyapunov exponents are non-zero.
\end{remark}

\begin{remark}
Because the attractor is hyperbolic and by Ruelle's classical result~\cite{Ruelle1997}, $\mu(\FF)$ depends on $\FF$ in a differentiable way.
\end{remark}

\begin{remark}
It is essential in our construction that $m \ge 2$. For technical reasons, we assume $m \ge 3$ but we believe all theorems are true for $m=2$ as well.
\end{remark}

\begin{remark}
The most of the construction is localized within a small ball $\hat A \subset M$, so the global properties of~$M$ are pretty much irrelevant to us. The compactness assumption about $M$ is used solely to apply the results of~\cite{Ilyashenko2012}.
\end{remark}

%\begin{remark}
%Theorem~\ref{t:endos} implies the robust existence of an orbit which is dense in $U$: due to the ergodicity of $\mu$, the orbits of a.e. point w.r.t $\mu$ are dense in $\supp\mu$. However, we can prove this directly, by constructing the proper ``critical words''.
%\end{remark}

%\begin{remark}
%Though our construction could be carried out for any $k \ge 1$, $m \ge 2$, the case $k = 1$, $m = 4$ allows to explain it with the least amount of technical details. Therefore in the Sections~\ref{s:skpr-circle}--\ref{s:diffeos}, we will stick to these values of $k$ and $m$.
%\end{remark}

To establish Theorem~\ref{t:endos}, first we prove it for a special class of skew products.
Let $h \colon \varphi \mapsto m\varphi \mod 1$ be the standard linear expanding map of $S^1$.
Denote by $\mathcal C (M)$ the space of all skew products over $h$ with the fiber~$M$, i.e., the maps of the form
\begin{equation}    \label{e:skew}
  F \colon (\varphi, x) \mapsto (h \varphi, f_\varphi(x)), \quad \varphi \in S^1, \quad x \in M.
\end{equation}
Here $f_\varphi$ is a $C^1$-diffeomorphism onto its image, $f_\varphi$ is $C^0$ in $\varphi$. The metric on $\mathcal C (M)$ is defined as
$$
\dist(F,G) := \sup_\varphi \dist_{C^1} \left( f_\varphi^{\pm 1}, g_\varphi^{\pm 1}\right).
$$

\begin{theorem} \label{t:skpr-circle}
There exists a nonempty open set in $\mathcal C(M)$ such that any skew product from this set has a massive attractor.
\end{theorem}

\begin{remark}
Theorem~\ref{t:skpr-circle} can be generalized in a straightforward way to the case when $h$ is any uniformly hyperbolic $m$-to-$1$ endomorphism of a smooth manifold, with $m \ge 3$.
\end{remark}

%The set~$\mathcal{O} \cap \mathcal{C}(M)$ is nonempty by construction, see Section~\ref{s:endo}. Thus Theorem~\ref{t:endos} implies Theorem~\ref{t:skpr-circle}.
In Sections~\ref{s:geometric}--\ref{s:strip-attr} we prove Theorem~\ref{t:skpr-circle} first, and then in Section~\ref{s:endo}, we apply Gorodetski-Ilyashenko-Negut techniques to obtain Theorem~\ref{t:endos}.

%\begin{remark}
%We also believe that the same statement is true for the skew products over any $q\colon \varphi \mapsto m\varphi \mod 1$, $m \ge 2$, $m \in \bbN$.
%\end{remark}

\section{Geometric construction}    \label{s:geometric}

In this Section, we establish two results which are not explicitly related to any dynamics, and might have other geometrical applications. We recently found a preprint by Homburg~\cite{Homburg2011} where he independently gives an argument equivalent to Subsection~\ref{ss:box}. However our Theorem~\ref{t:affine-maps} is stronger, and thus we give a complete proof here.

For regions $A$ and $B$, we will write $A \Supset B$ if $\inter A \supset B$. Note that this relation commutes with finite unions and intersections.

\begin{theorem} \label{t:affine-maps}
For any $n \in \bbN$ and $\eps > 0$ there exists a region $A \subset \bbR^n$ and a $C^\infty$-smooth arc of affine maps $E_t$, $t \in [0,1]$ such that
\begin{enumerate}
  \item\label{enum:am:1} for any $t \in [0,1]$ $E_t$ is contracting;
  \item\label{enum:am:2} $E_0 A \cup E_1 A \Supset A$;
  \item\label{enum:am:3} for any $t \in [0,1]$ $E_t$ is $\eps$-close to the identity in $\Aff(n)$.
\end{enumerate}
\end{theorem}

%\begin{remark}
%Theorem~\ref{t:affine-maps} remains true (and even easier) if we replace ``$\Supset$'' with just ``$\supset$''.
%\end{remark}

The proof is divided into two steps. The first step is to construct the region and the affine maps satisfying~(\ref{enum:am:1}) and~(\ref{enum:am:2}). The second step is to modify both the region and the maps to make them satisfy~(\ref{enum:am:3}).

\subsection{One box fits into its two smaller images}   \label{ss:box}

Fix some $0< \lambda < 1$ close enough to $1$. Let $B = [-r_1, r_1] \times [-r_2, r_2] \times \dots \times [-r_n, r_n] \subset \bbR^n$ be a standard rectangular box, and let the numbers $r_i > 0$ satisfy the inequalities
\begin{equation}\label{e:r-i}
\lambda r_i > r_{i+1}, \quad 1 \le i < n, \quad \text{and} \quad \lambda r_n > \frac12 r_1.
\end{equation}
Let $R \in SO(n)$, $R \colon (x_1, \dots, x_n) \mapsto ((-1)^{n+1}x_n, x_1, \dots, x_{n-1})$;
%send $e_i \mapsto e_{i+1}$, $1 \le i < n$, $e_n \mapsto e_1$:
%$$
%S = \begin{pmatrix}
%0 & 1 &  & \\
%& \ddots & \ddots & \\
%&  & \ddots & 1 \\
%1 & & & 0
%\end{pmatrix}
%$$
let $\Lambda = \lambda \cdot \Id$. Note that $\Lambda R B$ is the standard rectangular box with the sides $\lambda r_n, \lambda r_1, \dots, \lambda r_{n-1}$.

Due to \eqref{e:r-i} there exist two translations~$T_0$, $T_1$ of the space~$\bbR^n$ such that
\begin{equation}\label{e:tls}
\Lambda T_0 R B \cup \Lambda T_1 R B \Supset B.
\end{equation}
Indeed, the maps of the form~$(x_1 \pm s, x_2, \dots, x_n)$ suffice.

Let $T_t := tT_0 + (1-t)T_1$, $t \in [0,1]$.
Now the affine maps $G_t := \Lambda T_t R$ are contracting with the rate arbitrary close to identity, and $G_0 B \cup G_1 B \Supset B$. But the rotation $R$ is fixed and far from identity, so the statement~(\ref{enum:am:3}) cannot be achieved with these maps.

\subsection{Making the maps close to identity}

Now we construct new affine maps~$E_t$, $t \in [0,1]$, and a new region~$A$. The idea is to find~$E_t$ satisfying $E_t^k = G_t$ for $k \in \bbN$ big enough. Denote by $\bbR^+$ the space $\bbR$ as an additive group, denote by~$\Iso(n)$ the group of isometries of $\bbR^n$. The following Proposition is a simple exercise in linear algebra.

\begin{prop}
For $t\in [0,1]$ there exist Lie group homomorphisms $P_t\colon \bbR^+ \to \Aff(n)$ such that $P_t (1) = G_t$ and $\forall x\in \bbR$ $P_t(x) = \lambda^x \cdot U_t(x)$, $U_t(x) \in \Iso(n)$.
\end{prop}

Due to the homomorphism property, for any $t\in [0,1]$ $\lim_{k \to +\infty} P_t \left(\frac1k\right) = \Id$. Therefore, by choosing $k \in \bbN$ big enough, we can make $P_t \left(\frac1k\right)$ arbitrary close to identity. Now let $E_t := P_t \left(\frac1k\right)$. Note that
$$
E_t^k = \left(P_t \left(\frac1k\right)\right)^k = P_t\left(k \cdot \frac1k\right) = P_t(1) = G_t.
$$
Also note that $E_t$ is a contraction with the uniform rate $\lambda^{\frac1k} < 1$. By construction, the map $t \mapsto E_t$ is $C^\infty$-smooth.

Now for $i=0,1$ we can find regions~$B_i^{(0)}, \dots, B_i^{(k)}$ with the following properties:
\begin{enumerate}
  \item $B_i^{(0)} = B$;
  \item $E_i B_i^{(j)} \Supset B_i^{(j+1)}$;
  \item $B_0^{(k)} \cup B_1^{(k)} \Supset B$.
\end{enumerate}
We do this inductively: take $E_i B$, shrink it a little, denote the result by~$B_i^{(1)}$, then repeat $k-1$ times. This procedure works because of the gap in~\eqref{e:tls}.
%The former property holds because of~\eqref{e:tls}.

Let $A_i := \bigcup_{j=0}^{k-1} B_i^{(j)}$, $A := A_0 \cup A_1$.

\begin{prop}
$E_0 A \cup E_1 A \Supset A$.
\end{prop}

\begin{proof}
$$
E_0 A \cup E_1 A \supset E_0 A_0 \cup E_1 A_1 = E_0 \left(\bigcup_{j=0}^{k-1} B_0^{(j)}\right) \cup E_1 \left(\bigcup_{j=0}^{k-1} B_1^{(j)}\right) \Supset \bigcup_{j=1}^{k-1} B_0^{(j)} \cup \bigcup_{j=1}^{k-1} B_1^{(j)} \cup B = A.
$$
\end{proof}

Thus Theorem~\ref{t:affine-maps} is also proven.

\subsection{Nonlinear version}

\begin{theorem} \label{t:smooth-maps}
For any manifold~$M$ there exists a region $\hat A$ such that
for any $\eps > 0$ there exists a $C^\infty$-smooth arc of diffeomorphisms $\tilde f_t \colon [0,1] \times M \to M$ such that
\begin{enumerate}
  \item $\forall t \in [0,1]$ $\tilde f_t(\hat A) \Subset \hat A$;
  %REPLACED-2012-11-07%
  %\item $\forall t \in [0,1]$ $\norm{D \tilde f_t} < 1$; 
  \item $\forall t \in [0,1]$ $\forall x \in \hat A$ $\norm{D \tilde f_t (x)} < 1$;
  \item $\forall t \in [0,1]$ $\dist_{C^1} (\tilde f_t, Id) < \eps$;
  \item there exists a region $A \Subset \hat A$ such that $\tilde f_0(A) \cup \tilde f_1(A) \Supset A$.
\end{enumerate}
\end{theorem}

\begin{proof}
Take a Morse function~$H$ on $M$ with the global minimum $p$. Consider the gradient flow of $H$. There exists a small $\tau > 0$ such that the time-$\tau$ shift $S^\tau$ along the flow is $\eps$-close to the identity map~$\Id$.

In a neighborhood of $p$, the gradient flow of $H$ can be linearized to a diagonal form. Now we work in these linear coordinates.  Let $\hat A$ be a diffeomorphic ball with the center at $p$ such that $S^\tau(\hat A) \Subset \hat A$. Such a ball can be obtained using the Lyapunov function method.

%REPLACED-2012-11-07%
%Take $\eps_1 > 0$ such that being $\eps_1$-close to $\Id$ in the linear coordinates guarantees $\eps$-closeness to $\Id$ in the original coordinates. Take some smaller ball $A \Subset \hat A$ if needed. Apply Theorem~\ref{t:affine-maps} in $\check A$ to construct the affine maps $E_t\colon A \to A$, $t \in [0,1]$ that are $\eps_1$-close to $\Id$.
Take $\eps_1 > 0$ such that being $\eps_1$-close to $\Id$ in the linear coordinates guarantees $\eps$-closeness to $\Id$ in the original coordinates. Take some smaller ball $\check A \Subset \hat A$. Apply Theorem~\ref{t:affine-maps} in $\check A$ to construct a region~$A \subset \check A$ and the affine maps $E_t\colon \check A \to \check A$, $t \in [0,1]$ that are $\eps_1$-close to $\Id$ and $E_0(A) \cup E_1(A) \Supset A$.

%REPLACED-2012-11-07%
%Now glue together $S$ and $E_t$. Namely, we can find maps $\tilde f_t$ such that
Now glue together $S^\tau$ and $E_t$. Namely, we can find maps $\tilde f_t$ such that
\begin{enumerate}
  %REPLACED-2012-11-07%
  %\item $\left.\tilde f_t\right|_{M \setminus \hat A} = S^\tau$;
  %\item $\left.\tilde f_t\right|_{A} = E_t$.
  \item $\left.\tilde f_t\right|_{M \setminus \hat A} = S^\tau$;
  \item $\left.\tilde f_t\right|_{\check A} = E_t$.
\end{enumerate}
%REPLACED-2012-11-07%
%This is possible because, by construction, the maps $E_t$ are isotopic to $S^\tau$ in $A$.
This is possible because, by construction, the maps $E_t$ are isotopic to $S^\tau$ in $\check A$. By choosing~$\check A$ small enough, we can ensure all $\tilde f_t$ are close to identity.
\end{proof}

\section{The skew product}  \label{s:construction}

To prove Theorem~\ref{t:skpr-circle}, we provide a single skew product $F_0 \in \mathcal C(M)$, and show that every $F \in \mathcal C(M)$ which is close enough to $F_0$ has a massive attractor.
%
%For the skew products over the circle, the strategy is the same: we construct some special skew product $\hat F$ and prove that any $F \in \mathcal C(M)$ sufficiently close to $\hat F$ has a massive attractor.
%In this Section, we describe the construction and give the proof for $k=1$, $m=4$.

%Let $p = m-1$ (note that $p \ge k+1$, therefore it is eligible for Theorem~\ref{t:skpr-symbolic}), and fix some $f_1, \dots, f_p$ provided by Lemma~\ref{l:f1f2exist}.
%Now consider the following arcs in the base circle:
%Let $L_1, \dots, L_p \subset S^1$ be non-intersecting closed arcs, the length of each arc being $\frac1m$.
%$$
%L_1 := \left[0;\frac14\right], \ L_2 := \left[\frac12;\frac34\right].
%$$

%For every $\varphi \in L_i$, we let $f_\varphi := f_i$. On the rest of the circle, we define $f_\varphi(x)$ in an arbitrary way\footnote{Not really.}.
%%We just ensure that the resulting $\hat F$ is a N-S skew product.
%The resulting $\hat F$ could be made $C^\infty$-smooth in $\varphi$ and $x$, if needed. We need it to be~$C^2$ t obe able to apply the results from~\cite{ArXiv-Ilyashenko2010a}. Now $\hat F$ is well defined.

Let $L_0, L_1 \subset S^1$ be disjoint closed arcs, with the length of each arc equal to $\frac1m$. Then for $i=0,1$ we have $h(L_i) = S^1$. Fix a small $\eps > 0$ and apply Theorem~\ref{t:smooth-maps} to get the $C^\infty$-arc of diffeomorphisms $\tilde f_t\colon [0,1]\times M \to M$.

Now consider a $C^\infty$ map $l \colon S^1 \to [0,1]$ such that
$$
l|_{L_0} = 0, \quad l|_{L_1} = 1.
$$
Let $f_\varphi := \tilde f_{l(\varphi)}$ in~\eqref{e:skew}. This map~\eqref{e:skew} is the desired $C^\infty$ smooth skew product $F_0$. Actually, we only need it to be $C^2$ to apply the results from~\cite{Ilyashenko2012}. By construction, for any small enough $\eps$ both $F_0$ and any $F \in \mathcal C(M)$ close enough to $F_0$ satisfy the \emph{modified dominated splitting condition}~\cite{Ilyashenko2012}[Definition 2], which we rewrite in the form
\begin{equation}    \label{e:mdsc}
\max \left(\frac1m + \left\|\frac{\partial f_\varphi^{\pm 1}}{\partial \varphi}\right\|_{C^0}, \left\|\frac{\partial f_\varphi^{\pm 1}}{\partial x}\right\|_{C^0} \right) = L < m.
\end{equation}

\section{The lift to a solenoid}  \label{s:solenoid}

%Copy from IKS, page 459. !!!

\subsection{Solenoid skew products}

For any $F \in \mathcal C(M)$ we consider its solenoid extension $\hat F$. Here we follow~\cite{Ilyashenko2008}[p.~459]. Recall the definition of the solenoid map. For any~$R \ge 2$, $\alpha \ll 1$ let
$$
D:= \{z\in\bbC\mid|z|\le R\}, \quad B:=S^1 \times D,
$$
and consider the solenoid map $H \colon B \to B$
$$
H\colon (\varphi,z) \mapsto (h \varphi, e^{2\pi i \varphi} + \alpha z), \quad \varphi \in S^1, \quad z \in D.
$$

For proper~$R, \alpha$ (depending only on $m$), $H$ is a 1-1 map onto its image. We fix these values of $R, \alpha$ and will never need them any more. The maximal attractor of $H$,
$$
\Lambda := \bigcap_{n \ge 0} H^n B,
$$
is called a \emph{Smale-Williams solenoid}. This set is invariant, locally maximal, and hyperbolic. The restriction~$H|_\Lambda$ is invertible.

Given $F \colon S^1 \times M \to S^1 \times M$, we define~$\hat F \colon B \times M \to B \times M$ by the following formula:
$$
\hat F \colon (b, x) \mapsto (H b, f_\varphi(x)), \quad b := (\varphi, z) \in B.
$$
%\begin{remark}
$\hat F$ is a skew product over~$H$ with fiber maps not depending on $z$. Thus it can be also viewed as a skew product over $h$, with the fibers~$D \times M$. Denote by~$\pi$ the projection along~$z$; $\pi \colon (\varphi, z, x) \mapsto (\varphi,x)$. Then~$F \circ \pi = \pi \circ \hat F$.
%\end{remark}

Obviously $\hat F$ has $\Lambda \times M$ as an invariant set. For $f_\varphi^{\pm 1}$ uniformly $C^1$-close to identity, $\Lambda \times M$ is partially hyperbolic with $\dim E^s = 2$, $\dim E^u = 1$, $\dim E^c = \dim M$.

\subsection{Generic endomorphisms of $S^1 \times M$}

The same construction works for any $\mathcal F \in C^1(X)$. Namely, for $(z,\varphi,x) \in B \times M$ let
$$
\hat{\mathcal F} (z,\varphi,x) = (e^{2\pi i \varphi} + \alpha z, \mathcal F (\varphi,x)).
$$
Note that for $C^1$-close $\FF,\GG$ we get $C^1$-close extensions $\hat\FF, \hat\GG$.

\section{Sink skew products and their attractors}  \label{s:two-sided-attractor}
%
%North-South maps are the simplest hyperbolic diffeomorphisms of a sphere~$S^k$. They have exactly two fixed points, one is repelling (North Pole) and another is attracting (South Pole), both fixed points are hyperbolic. Following~\cite{Ilyashenko2010}, we give a generic definition of North-South skew products in the following way.

Let $q\colon B \to B$ be a dynamical system and $F \colon B \times M \to B \times M$ be a skew product over $q$.

\begin{definition}  \label{d:N-S-skpr}
A skew product $F$ is called a \emph{sink skew product} if there exists a \emph{sink region}~$S\subset M$ diffeomorphic to a closed ball~$D^n$ such that $\forall b \in B$ the fiber map $f_b$ has the following properties:
\begin{enumerate}
  \item\label{enum:3:1} $f_b$ is a Morse-Smale diffeomorphism;
  \item\label{enum:3:2} $\exists! \,x \in \Fix (f_b)$ within $\inter S$ and this $x$ is an attracting fixed point;
%  \item\label{enum:3:3} all the repellers of $f_b$ lie strictly inside $R$;
  \item\label{enum:3:4} $f_b$ brings $S$ strictly into itself and is contracting on $S$ uniformly in $b$;
  \item\label{enum:3:5} all inverse maps $f_b^{-1}$ are expanding on $S$ uniformly in $b$;
%  \item\label{enum:3:6} all inverse maps $f_b^{-1}$ bring $R$ strictly into itself and are contracting on $R$ uniformly in $b$;
%  \item\label{enum:3:7} all maps $f_b$ are expanding on $R$ uniformly in $b$;
  \item\label{enum:3:8} the maps $f_b$ and $f_b^{-1}$ depend continuously on $b$ in the $C^0$ topology.
\end{enumerate}
\end{definition}

%We will often abbreviate North-South as N-S.

%\begin{remark}[for myself!!!] Is $C^0$ enough? Maybe we need at least $C^1$ here. \textbf{Answer:} Yes, because HPS gives only $C^0$, and Gorodetski adds H\"older continuity only.
%\end{remark}

%\begin{definition}
%A N-S skew product is called \emph{A-affine} if all the restrictions
%$$
%f_b|_A, \quad f_b^{-1}|_A %, \quad f_b|_R, \quad f_b^{-1}|_R
%$$
%are affine maps (in the local coordinates of the cube~$A$).
%\end{definition}
%
%Additionally, for a smooth uniformly hyperbolic dynamics $q\colon B\to B$ in base we can define the notion of partial hyperbolicity. Let $\lambda^-$ and $\lambda^+$ be the uniform rates of the contraction and expansion of the base dynamics.
%
%\begin{definition}
%A N-S skew product is called \emph{partially hyperbolic} if the fiber dynamics is dominated by the base dynamics: there exists $\eps > 0$ such that
%$$
%\lambda^- + \eps < \abs{(f_b^{\pm 1})'} < \lambda^+ - \eps \quad \forall b \in B.
%$$
%\end{definition}

Apparently, as the region~$B \times S$ is trapping (property~\ref{enum:3:4}), it admits a maximal attractor $\Amax = \bigcap_{n \ge 0} F^n (B \times S)$. It is a closed invariant set: $F^{-1} (\Amax) = \Amax$. Now assume $q$ is invertible. In this case, $\Amax$ has a really simple form.

\begin{lemma}   \label{l:two-sided-attr}
    Let $F$ be a sink skew product over an invertible $q\colon B\to B$. Let $P$ be an ergodic invariant measure on $B$. Consider $\Amax := \bigcap_{n=0}^{\infty} F^n (B \times S)$. Then
    \begin{enumerate}
        \item $\Amax$ is the graph $\Gamma$ of a continuous function $\gamma = \colon B\to S$. The projection $j|_\Gamma \colon \Gamma \to B$ is a bijection. Under this bijection, $F|_\Gamma$ is conjugated to the dynamics in the base.
        \item\label{enum:2:2} $\Amin = \Gamma$.
        \item There exists an ergodic SRB measure $\mu_\infty$ in $B \times S$. Its basin contains $B\times S$. This measure is supported on~$\Gamma$ and is precisely the pull-back of the measure~$P$ under the bijection $j|_\Gamma \colon \Gamma \to B$. Therefore, $\supp\mu_\infty = \Gamma$.
    \end{enumerate}
\end{lemma}

\begin{remark}
The statement~(\ref{enum:2:2}) means that all reasonable definitions of an attractor coincide with $\Gamma$, see~\cite{Ilyashenko2010}(12). In particular, $\Amin = \Astat = A_M = \Amax$.
\end{remark}

Lemma~\ref{l:two-sided-attr} is rather elementary, and appears to be folklore. The proof can be found, for instance, in~\cite{Ilyashenko2010}[Subsections 2.3--2.5]. Though \cite{Ilyashenko2010} state their results only for $B = \Sigma$\footnote{$\Sigma$ is the bi-infinite analog of $\Sigma^+$. The shift~$\sigma\colon\Sigma\to\Sigma$ is invertible.}, the same argument works perfectly for any invertible $q\colon B\to B$ in the base. Also, formally speaking, they deal not with sink skew products, but with their specific subset, so-called North-South skew products. But the same proofs extend easily to sink skew products.
%In the following Lemma~\ref{l:two-sided-attr}, we summarize the relevant results from~\cite{Ilyashenko2010}.

%REPLACED-2012-11-07%
%By construction, any $F$ close enough to $F_0$ and any $\hat F$ close enough to $\hat F_0$ are sink skew products with $S = \check A$. Therefore Lemma~\ref{l:two-sided-attr} applies to $\hat F$.
By construction, any $F$ close enough to $F_0$ and any $\hat F$ close enough to $\hat F_0$ are sink skew products with $S = \hat A$. Therefore Lemma~\ref{l:two-sided-attr} applies to $\hat F$.

\section{Projection of the attractor is dense within a strip} \label{s:density}

%\subsection{The projection of the attractor}   \label{s:attractor-projection}

Note that the projection~$\mu = \pi_* \mu_\infty$ of the SRB measure~$\mu_\infty$ for $\hat F$ is an SRB measure for $F$; $\supp\mu = \pi\Gamma$.
%\subsection{Proof OF WHAT?}
%
%Take any $F \in \mathcal C(M)$ which is close enough to $\hat F$. Note that $h(L_i) = S^1$, $i=1,\dots,p$, and also for any $\varphi_i \in L_i$ the map $f_{\varphi_i}$ is $C^1$-close to $f_i$.
%
%Thus we can literally repeat the proof we gave in Section~\ref{s:density}. Again, for any $\varphi \in S^1$ we find $\varphi_1 \in L_1, \dots, \varphi_p \in L_p$ such that $h(\varphi_i) = \varphi$. As $\varphi_i \in L_i$, the maps $f_{\varphi_i}$ are close to $f_i$.
%
%Theorem~\ref{t:skpr-circle} is proven.
Take $A \Subset S$ from Theorem~\ref{t:smooth-maps}.
\begin{prop} \label{p:strip}
For any $F \in \mathcal C(M)$ which is $C^1$-close enough to $F_0$ we have
\begin{equation}
S^1 \times A \subset \pi \Gamma.
\end{equation}
\end{prop}

Let $\frac12 < \lambda < 1$ be the uniform upper rate of contraction of the fiber maps in the sink~$S$:
$$
\lambda = \sup_{\varphi \in S^1, \ x \in S} \norm{Df_\varphi(x)}.
$$
We also denote by $S_\varphi$ and $A_\varphi$ the regions~$S$ and $A$ in the fiber over $\varphi$:
$$
S_\varphi := \{ \varphi \} \times S, \ A_\varphi := \{ \varphi \} \times A,
$$
and let $\Gamma_\varphi := S_\varphi \cap \pi\Amax$. We write $\abs{A} := \diam A$.

The proof will be carried out by induction. Namely, we will show that for any $\varphi \in S^1$ and any $n \in \bbN$ there exists a finite covering of $A_\varphi$ by the balls of radius $\lambda^{n-1} \cdot \abs{A}$ whose centers lie in the set $\Gamma_\varphi$. As $n$ could be taken arbitrarily large, this will prove the density of $\Gamma_\varphi$ in $A_\varphi$ for any $\varphi \in S^1$ and thus the density of $\pi\Amax$ in $S^1\times A$. Finally, the fact that $\pi\Amax$ has to be a closed set finishes the proof of the Proposition.

\subsection*{Base: $n = 1$}
Any point from the non-empty set $\Gamma_\varphi$ suffices as the center of a single ball of radius $\abs{A}$, which clearly contains $A_\varphi$.

\subsection*{Step} Assume that for some $n \in \bbN$ and for any $\varphi \in S^1$, there exists a finite covering of $A_\varphi$ by the balls of the radius $\lambda^{n-1} \cdot \abs{A}$ with centers in the set $\Gamma_\varphi$. Denote by~$O_\varphi$ the set of all centers of the balls.

Take any $\varphi' \in S^1$. There exist two preimages of $\varphi'$, $\varphi_0'$ and $\varphi_1'$, within $L_0$ and $L_1$ respectively:
$$
h \varphi_i' = \varphi', \quad \varphi_i' \in L_i.
$$

Because of the invariance of $\Amax$ and $\pi\Amax$, we have
$$
\Gamma_{\varphi'} = \bigcup_{\varphi \in h^{-1} \varphi'} f_{\varphi} (\Gamma_{\varphi}) \supset f_{\varphi_0'} (\Gamma_{\varphi_0'}) \cup f_{\varphi_1'} (\Gamma_{\varphi_1'}).
$$

We know that every $f_\varphi$ contracts $S_\varphi$ with the uniform upper rate of $\lambda<1$. This observation, combined with the induction assumption, gives us that the balls of radius $\lambda^n \cdot \abs{A}$ with the centers in $f_{\varphi} (O_{\varphi})$ constitute a covering of the regions $f_{\varphi} (A_{\varphi}) \subset S_{\varphi'}$, $\varphi \in h^{-1} \varphi'$.
%we may take the images of the centers of the balls $f_{i\varphi'} (\Gamma_{i\varphi'})$ contains a $l^n \cdot \abs{A}$-net for the segment $f_{i\varphi'} (A_{i\varphi'}) \subset A_{\varphi'}$, $i = 1, 2$.
Now note that the maps $f_{\varphi_i'}$ are $C^1$-close to the maps $f_i$ originally generated by Theorem~\ref{t:smooth-maps}, and therefore
\begin{equation}    \label{e:per-incl}
A \subset f_{\varphi_0'} (A) \cup f_{\varphi_1'} (A),
\end{equation}
%see (\ref{enum:1:2}) from Lemma~\ref{l:f1f2exist}.
Thus the balls of radius $\lambda^n \cdot \abs{A}$ with the centers in $f_{\varphi_0'} (O_{\varphi_0'}) \cup f_{\varphi_1'} (O_{\varphi_1'})$ cover the whole region $A$.
%$\Gamma_{\varphi'}$ contains a $l^n \cdot \abs{A}$-net for the whole segment $A$.

The induction step is complete, and Proposition~\ref{p:strip} is proven.

%\subsubsection*{Case 2: $m = 2$ and $\varphi' \in q_2 TR$.} This tiny case is more technical to deal with. The choice of transition arcs $TR$ guarantees that for any $0 < n \le N$, $q_2^{-n} \varphi' \cap TR = \emptyset$.

\section{From a strip to the whole attractor}    \label{s:strip-attr}

In this section we finish the proof of Theorem~\ref{t:skpr-circle}. From section~\ref{s:density} we know that
$$
S^1 \times A \subset \supp\mu.
$$
As $\mu$ is ergodic, its support is a transitive invariant set. Therefore
$$
\supp\mu = \bigcup_{n = -\infty}^{+\infty} F^n (S^1 \times A).
$$
For any finite $N \in \bbN$, the set $\bigcup_{n = -N}^{N} F^n (S^1 \times A)$ is the closure of a non-empty connected open set. Thus this is true for the whole $\supp \mu$. What remains to be proven is that there exists an open $U \supset \supp\mu$ such that
$$
\supp\mu = \bigcap_{n=0}^{+\infty} F^n (U).
$$

In the solenoid extension~$\hat F$, the set $\Gamma$ is a hyperbolic attractor with a uniform rate of contraction along the fibers. For a small $\eps > 0$ let
$$
\hat U_\eps := \{ (z,\varphi,x) \mid \dist_M(x,\gamma(z,\varphi)) < \eps \}.
$$
This set is open, and $U := \pi \hat U_\eps$ is an open neighborhood of $\supp\mu$. In the solenoid extension,
$$
\Gamma = \bigcap_{n=0}^{+\infty} \hat F^n \left(\hat U_\eps\right).
$$
Thus
$$
\supp\mu = \pi\Gamma = \bigcap_{n=0}^{+\infty} \pi \hat F^n \left(\hat U_\eps\right) = \bigcap_{n=0}^{+\infty} F^n \left(U\right).
$$
The proof of Theorem~\ref{t:skpr-circle} is complete.

\section{Perturbation in the space of all smooth endomorphisms}    \label{s:endo}

In this section we prove our main Theorem~\ref{t:endos}. Consider any map~$\FF \in C^1(X)$ which is $C^1$-close to $F_0$.
We want to show that $\FF$ has a massive attractor.

Note that though $F_0$ is a skew product, $\FF$ is usually not. The straight vertical foliation is no longer invariant, so the skew product structure is lost. However, for small enough perturbations of $F_0$, we are able to rectify the perturbed map and recover the structure. Here the  partially hyperbolic techniques from~\cite{HPS1977}, \cite{Gorodetskiui2006}, \cite{Ilyashenko2012} come into play.

First of all, we consider the solenoidal extensions of $F_0$ and $\FF$. Denote them by~$\hat F_0$ and $\hat \FF$. Their fiber maps are independent of~$z$. According to Section~\ref{s:solenoid}, they are $C^1$-close. In the rest of this section we use the notations of Section~\ref{s:solenoid}.

We now invoke the following theorem from~\cite{Ilyashenko2008}[Theorem~5] (which is in fact a modified version of theorems~\cite{Ilyashenko2012}[Theorem~2, Theorem~1]):
\begin{theorem} \label{t:iks}
Suppose that the fiber maps of $\hat F_0$ satisfy the modified dominated splitting condition~\eqref{e:mdsc}.
Then there exists a $\rho>0$ such that any map $\GG$ that is $\rho$-close to $\hat F_0$
in $C^1$,
$$
\dist_{C^1}(\hat F_0^{\pm 1},\GG^{\pm 1}) \le \rho,
$$
has the following properties:

(a) There exists a semiconjugacy $p \colon S^1 \times D \times M \to S^1$ such that the diagram
$$
\begin{CD}
S^1 \times D \times M @>{\GG}>> S^1 \times D \times M \\
@VpVV @VVpV \\
S^1 @>h>> S^1 \\
\end{CD}
$$
commutes. Moreover, the map
$$
\hat\HH \colon S^1 \times D \times M \to S^1 \times D \times M, \quad \hat\HH (\varphi,z,x) := (p(\varphi,z,x),z,x)
$$
is a homeomorphism.

(b) The fibers of $p$,
$$
\tilde N_\varphi := p^{-1}(\varphi),
$$
are the graphs of smooth maps
$$
\tilde\beta_\varphi \colon D \times M \to S^1.
$$
The maps $\tilde\beta_\varphi$ are
% H\"older
continuous in $\varphi$.
%$$
%\abs{\tilde\beta_\varphi - \tilde\beta_{\varphi'}} \le C |\varphi - \varphi'|^\alpha.
%$$
\end{theorem}
Actually, the maps $\tilde\beta_\varphi$ are even H\"older continuous in $\varphi$ but we do not use this fact.

The following statement is proven in Subsection 5.2 of~\cite{Ilyashenko2008}.
\begin{theorem} \label{t:iks52}
In the setting of Theorem~\ref{t:iks}, suppose $\GG = \hat\FF$ is a solenoidal extension of $\FF$ that is $\rho$-close to $F_0$. Then the maps~$p(\varphi, z, x)$ and $\tilde\beta_\varphi(z,x)$  are independent of~$z$.
\end{theorem}

It follows from Theorem.~\ref{t:iks} that we can rectify the perturbed skew product $\GG$. Namely, let $\hat F = \hat\HH \circ \hat\FF \circ \hat\HH^{-1}$. Statement a) of Theorem~\ref{t:iks} implies that $\hat F$ is a solenoidal skew product.

By Theorem~\ref{t:iks52}, the map~$\HH \colon S^1 \times M \to S^1 \times M$ such that $\HH(\varphi,x) = (p(\varphi,x),x)$ is a homeomorphism. Thus~$\hat F$ is the solenoidal extension of a circle skew product~$F := \HH \circ \FF \circ \HH^{-1}$.

%there exists a H\"older homeomorphism $\mathcal H \colon X \to X$ which conjugates $\hat \FF$ to a new solenoidal skew product $\hat F$.

\cite{Ilyashenko2012} also provides an important addition which we apply here:
\begin{theorem} \label{t:in}
The fiber maps $f_b$ of the skew product $F$ are $C^1$-close to those of the skew product $F_0$, in the following sense:
$$
\dist(f_b^{\pm 1}, f_{(0),b}^{\pm 1})_{C^1} \le O(\rho).
$$
\end{theorem}
This is the statement~(12), proved in Subsection~6.2 of \cite{Ilyashenko2012}.

Thus $F$ is $C^1$-close to $F_0$. By Theorem~\ref{t:skpr-circle}, $F$ has a massive attractor. Now note that the property to have a massive attractor is preserved by a homeomorphic conjugacy. Theorem~\ref{t:endos} is proven.

\section{Acknowledgements}

The author would like to thank Yulij Ilyashenko for posing the problem of the robustness of thick attractors, Victor Kleptsyn, Anton Gorodetski, Stefano Luzzatto and Jos\'{e}~F.~Alves for fruitful discussions, and the anonymous referee for helpful comments. The author is grateful to Institut~de~Recherche Mathematique~de~Rennes, where the idea of this paper was born. The author says special thank you to catamaran ``Barbos'' where all the pieces of the proof finally came together.
%Universit\'e~de~Rennes~1 and

%\section{Trash}

\bibliographystyle{plain}
%\bibliography{dynsys}

\begin{thebibliography}{10}

\bibitem{Abdenur2004}
Flavio Abdenur, Christian Bonatti, and Lorenzo~J. D{\'{\i}}az.
\newblock Non-wandering sets with non-empty interiors.
\newblock {\em Nonlinearity}, 17(1):175--191, 2004.

\bibitem{Alves2002}
Jos{\'e}~F. Alves and Marcelo Viana.
\newblock Statistical stability for robust classes of maps with non-uniform
  expansion.
\newblock {\em Ergodic Theory Dynam. Systems}, 22(1):1--32, 2002.

\bibitem{Avila2006}
Artur Avila, S{\'e}bastien Gou{\"e}zel, and Masato Tsujii.
\newblock Smoothness of solenoidal attractors.
\newblock {\em Discrete Contin. Dyn. Syst.}, 15(1):21--35, 2006.

\bibitem{Bowen1975a}
Rufus Bowen.
\newblock {\em Equilibrium states and the Ergodic theory of {A}nosov
  diffeomorphisms}, volume 470 of {\em Lecture Notes in Mathematics}.
\newblock Springer-Verlag, Berlin, 1975.

\bibitem{Brin1975}
Michael Brin.
\newblock Topological transitivity of a certain class of dynamical systems, and
  flows of frames on manifolds of negative curvature.
\newblock {\em Funkcional. Anal. i Prilo\v zen.}, 9(1):9--19, 1975.

\bibitem{Melo1993}
Welington de~Melo and Sebastian van Strien.
\newblock {\em One-Dimensional Dynamics}.
\newblock Springer-Verlag, Berlin, 1993.

\bibitem{Dobbs2006}
Neil Dobbs.
\newblock Hyperbolic dimension for interval maps.
\newblock {\em Nonlinearity}, 19(12):2877, 2006.

\bibitem{Fisher2006a}
Todd Fisher.
\newblock Hyperbolic sets with nonempty interior.
\newblock {\em Discrete Contin. Dyn. Syst.}, 15(2):433--446, 2006.

\bibitem{Gorodetskiui2006}
Anton Gorodetski.
\newblock Regularity of central leaves of partially hyperbolic sets and
  applications.
\newblock {\em Izv. Ross. Akad. Nauk Ser. Mat.}, 70(6):19--44, 2006.

\bibitem{Gorodetski2012}
Anton Gorodetski.
\newblock On stochastic sea of the standard map.
\newblock {\em Comm. Math. Phys.}, 309(1):155--192, 2012.

\bibitem{HPS1977}
Morris~W. Hirsch, Charles~C. Pugh, and Michael Shub.
\newblock {\em Invariant Manifolds (Lecture Notes in Mathematics)}.
\newblock Springer, 1977.

\bibitem{Homburg2011}
Ale~Jan Homburg.
\newblock Robustly minimal iterated function systems on compact manifolds
  generated by two diffeomorphisms.
\newblock {\em Preprint}, 2011.

\bibitem{Ilyashenko2011}
Yulij Ilyashenko.
\newblock Thick attractors of boundary preserving diffeomorphisms.
\newblock {\em Indagationes Mathematicae}, 22(3-4):257--314, 2011.

\bibitem{Ilyashenko2008}
Yulij Ilyashenko, Victor Kleptsyn, and Petr Saltykov.
\newblock Openness of the set of boundary preserving maps of an annulus with
  intermingled attracting basins.
\newblock {\em Journal of Fixed Point Theory and Applications}, 3(2):449--463,
  September 2008.

\bibitem{Ilyashenko2010}
Yulij Ilyashenko and Andrei Negut.
\newblock Invisible parts of attractors.
\newblock {\em Nonlinearity}, 23(5):1199--1219, 2010.

\bibitem{Ilyashenko2012}
Yulij Ilyashenko and Andrei Negut.
\newblock {H}\"{o}lder properties of perturbed skew products and {F}ubini
  regained.
\newblock {\em Nonlinearity}, 25(8):2377, 2012.

\bibitem{McCluskey1983}
Heather McCluskey and Anthony Manning.
\newblock Hausdorff dimension for horseshoes.
\newblock {\em Ergodic Theory and Dynamical Systems}, 3:251--260, 1983.

\bibitem{Rams2003}
Micha{\l} Rams.
\newblock Absolute continuity of the {SBR} measure for non-linear fat baker
  maps.
\newblock {\em Nonlinearity}, 16(5):1649--1655, 2003.

\bibitem{Ruelle1997}
David Ruelle.
\newblock Differentiation of {SRB} states.
\newblock {\em Communications in Mathematical Physics}, 187:227--241, 1997.

\bibitem{Tsujii2001}
Masato Tsujii.
\newblock Fat solenoidal attractors.
\newblock {\em Nonlinearity}, 14(5):1011--1027, 2001.

\bibitem{Viana1997}
Marcelo Viana.
\newblock Multidimensional nonhyperbolic attractors.
\newblock {\em Publications Math{\'e}matiques de L'IH{\'E}S}, 85(1):63--96,
  December 1997.

\end{thebibliography}

\end{document}